\documentclass[preprint,12pt]{elsarticle}

\usepackage{graphicx}
\usepackage{amsmath,amssymb,amsfonts}
\usepackage{amsthm}
\usepackage{mathtools}
\usepackage{mathrsfs}
\usepackage{xurl}
\usepackage{hyperref}

\newtheorem{theorem}{Theorem}[section]
\newtheorem{proposition}[theorem]{Proposition}

\newtheorem{corollary}[theorem]{Corollary}

\theoremstyle{definition}

\theoremstyle{remark}
\newtheorem{remark}[theorem]{Remark}

\newcommand{\HH}{\mathbb{H}}
\newcommand{\ZZ}{\mathbb{Z}}

\newcommand{\GL}{\mathrm{GL}}

\begin{document}
	\begin{frontmatter}
\title{\bfseries
	Period Relations for Theta Products and Elliptic Integral Moments
}

\author[temple]{Dianbin Bao}
\ead{tud53299@temple.edu}

\affiliation[temple]{
	organization={Temple University},
	city={Philadelphia},
	state={PA},
	country={USA}
}

\begin{abstract}
We construct period polynomial relations for finite systems of theta products stable under modular transformations and use them to derive identities among critical $L$-values. Our main example is a three-component theta system of weight $5$, whose coupled period polynomials are determined explicitly and yield new cross-form relations among critical values of the associated eta products. These relations are not consequences of the functional equations of the individual forms. We also develop a weight-$4$ system arising from a quadratic twist of conductor $3$ and obtain relations linking critical values of the original and twisted modular forms. Through modular parametrizations by complete elliptic integrals, these $L$-value identities yield corresponding moment identities. The method gives a unified modular-symbol explanation of several previously known elliptic integral moment relations while producing new critical-value relations between distinct modular forms.
\end{abstract}
	
\begin{keyword}period polynomials\sep critical $L$-values\sep theta products\sep
modular forms\sep elliptic integral moments
\end{keyword}


\end{frontmatter}
	\section{Introduction}
Moments of complete elliptic integrals were studied extensively by Wan
\cite{MR2845511}, and related identities were investigated by Zhou
\cite{MR3231318} using methods involving Legendre functions and spherical
rotations. Connections between these moments and critical values of modular $L$-functions
and lattice sums were developed further by Rogers, Wan, and Zucker
\cite{MR3338042} and by Wan and Zucker \cite{MR3490555}. A recurring
phenomenon in this work is that several elliptic integral moments can be
identified with critical values of modular $L$-functions, while the
corresponding critical values themselves satisfy relations involving rational
multiples of powers of $\pi$. 

The contribution of the present paper is not to rederive these moment
evaluations analytically, but to identify a modular mechanism that explains
why such relations occur. The key point is to apply period-polynomial methods
simultaneously to finite systems of modular forms that are permuted by modular
transformations. Instead of considering the period polynomial of a single
form in isolation, we study coupled period vectors associated with several
theta or eta products. Modular transformations and modular-symbol
decompositions then impose relations between these period vectors and,
consequently, between critical values of distinct modular forms. This produces
cross-form critical value identities that are not simply consequences of the
functional equation of any one of the forms involved. In the examples arising
from elliptic integral moments, this gives a conceptual interpretation of
known identities as consequences of modular symmetry; in the quadratic-twist
example considered below, the same method also yields explicit cross-relations
between the critical values of two different modular forms.

The period-theoretic background is classical. Manin's work on modular symbols
and periods \cite{MR314846,MR345909} and Shimura's theory of periods of modular
forms \cite{MR463119} provide a general framework for the algebraic structure
of critical values. What is specific to the present work is the use of this
framework for finite-dimensional systems of modular forms rather than for a
single form. We identify systems of theta and eta products that are stable
under suitable slash operations and combine their transformation laws with
explicit modular-symbol decompositions. The resulting coupled relations among
period polynomials determine relations among the corresponding critical
$L$-values directly from modular symmetry.

This mechanism appears in several concrete forms. For certain theta-product
systems, a finite collection of modular forms is stable under the slash
actions of the standard generators $S:\tau\mapsto-1/\tau$ and
$T:\tau\mapsto\tau+1$, and the resulting Manin relations determine the
associated period vectors. In other examples, Hecke or Atkin--Lehner
symmetries select the relevant combinations of periods. For pairs of modular
forms related by twisting, suitable slash transformations together with
modular-symbol decompositions between appropriate cusps produce direct
relations between the period polynomials of the two forms, and hence between
their critical values. Thus the various examples treated below are instances
of a common principle: modular symmetry acting on a finite system of forms can
force relations among critical values belonging to different $L$-functions.

\subsection{Overview of the results}

In Section~\ref{three-theta} we develop this mechanism for a three-component
system of weight-$5$ theta products. The corresponding period polynomials are
determined from the actions of $S$ and $T$ together with modular symbol
relations, and their eta-product realizations produce cross-form
critical value relations and give a conceptual explanation of several
elliptic integral identities previously studied by Wan \cite{MR2845511},
Rogers, Wan, and Zucker \cite{MR3338042}, and Zhou \cite{MR3231318}.

Section~\ref{weight-4} treats a two-component system of weight-$4$ eta
quotients. A modular-symbol decomposition through the cusp $\frac{1}{4}$ relates the
period polynomials of the two forms and explains the resulting cross relations
among their critical values.

In Section~\ref{conductor3} we apply the same viewpoint to the weight-$4$
level-$6$ form
\[
\phi(\tau)=\eta(\tau)^2\eta(2\tau)^2\eta(3\tau)^2\eta(6\tau)^2
\]
and its quadratic twist $\psi=\phi\otimes\chi_{-3}$. The conductor-$3$
translation and Fricke symmetry generate a small slash-stable system whose
cusp orbit forms an ideal quadrilateral. The resulting modular symbol
relation yields new cross relations between critical values of $\phi$ and
$\psi$, which in turn connect with known Bessel-moment evaluations.

Finally, Section~\ref{sec:general-framework} formulates the common
modular-symbol framework underlying these examples.

\section{Preliminaries}
	\label{sec:preliminaries}
Throughout the paper, let $\tau\in\HH$ and set
\[
	q=e^{2\pi i\tau},\qquad \mathfrak q=e^{\pi i\tau}=q^{1/2}.
\]
Let
	\[
	\eta(\tau)=q^{1/24}\prod_{n=1}^{\infty}(1-q^n)
	\]
be the Dedekind eta function. We shall use the standard transformation laws
	\[
	\eta(\tau+1)=e^{\pi i/12}\eta(\tau),
	\qquad
	\eta\left(-\frac{1}{\tau}\right)=\sqrt{-i\tau}\,\eta(\tau).
	\]
For a cusp form
	\[
	f(\tau)=\sum_{n=1}^{\infty}a_nq^n,
	\]
we denote its $L$-function by
	\[
	L(f,s)=\sum_{n=1}^{\infty}\frac{a_n}{n^s}.
	\]
	
	\subsection{Theta functions and elliptic parametrization}
The classical theta constants are
	\begin{align*}
		\theta_2(\tau)&=\sum_{n\in\ZZ}\mathfrak q^{(n+\frac{1}{2})^2},\\
		\theta_3(\tau)&=\sum_{n\in\ZZ}\mathfrak q^{n^2},\\
		\theta_4(\tau)&=\sum_{n\in\ZZ}(-1)^n\mathfrak q^{n^2}.
	\end{align*}
	
For the standard generators
	\[
	S=\begin{pmatrix}0&-1\\1&0\end{pmatrix},
	\qquad
	T=\begin{pmatrix}1&1\\0&1\end{pmatrix},
	\]
the theta constants satisfy
	\begin{align*}
		\theta_2(-1/\tau)&=(-i\tau)^{1/2}\theta_4(\tau),
		&
		\theta_2(\tau+1)&=e^{\pi i/4}\theta_2(\tau),\\
		\theta_3(-1/\tau)&=(-i\tau)^{1/2}\theta_3(\tau),
		&
		\theta_3(\tau+1)&=\theta_4(\tau),\\
		\theta_4(-1/\tau)&=(-i\tau)^{1/2}\theta_2(\tau),
		&
		\theta_4(\tau+1)&=\theta_3(\tau).
	\end{align*}
We also use the standard theta--eta identities
	\begin{align*}
		\theta_2(\tau)&=2\frac{\eta(2\tau)^2}{\eta(\tau)},\\
		\theta_3(\tau)&=\frac{\eta(\tau)^5}
		{\eta(\tau/2)^2\eta(2\tau)^2},\\
		\theta_4(\tau)&=\frac{\eta(\tau/2)^2}{\eta(\tau)}.
	\end{align*}
For $0<k<1$, let
	\[
	K(k)=\int_0^{\pi/2}
	\frac{d\theta}{\sqrt{1-k^2\sin^2\theta}},
	\qquad
	K'(k)=K\bigl(\sqrt{1-k^2}\bigr).
	\]
For the classical theta-function parametrization of complete elliptic integrals, see Borwein and Borwein \cite{MR1641658}. The elliptic nome is
	\[
	\mathfrak q=e^{-\pi K'(k)/K(k)}=e^{\pi i\tau},
	\]
	and the classical theta parametrization gives
	\[
	K(k)=\frac{\pi}{2}\theta_3(\tau)^2,
	\]
	together with
	\[
	k=\frac{\theta_2(\tau)^2}{\theta_3(\tau)^2},
	\qquad
	\sqrt{1-k^2}=\frac{\theta_4(\tau)^2}{\theta_3(\tau)^2}.
	\]
Finally, differentiation of the nome gives
\[
	dk=\frac{2k(1-k^2)K(k)^2}
	{\pi^2\mathfrak q}\,d\mathfrak q.
\]

\subsection{Period polynomials}
We recall the standard period polynomial formalism for cusp forms. For background, see Manin \cite{MR314846,MR345909} and Shimura \cite{MR463119}.
Let $f(\tau)=\sum_{n=1}^{\infty}a_nq^n$ be a cusp form of integral weight $\kappa$. Its period polynomial is defined by 
\[
r_f(X):=\int_0^{i\infty}f(\tau)(\tau-X)^{\kappa-2}\,d\tau.
\]
For $0\leq j\leq \kappa-2$, the Mellin transform gives
\[
\int_0^{i\infty}f(\tau)\tau^j\,d\tau
=i^{j+1}\frac{\Gamma(j+1)}{(2\pi)^{j+1}}L(f,j+1),
\]
and hence
\[
r_f(X)=\sum_{j=0}^{\kappa-2}
(-1)^{\kappa-2-j}\binom{\kappa-2}{j}
i^{j+1}\frac{\Gamma(j+1)}{(2\pi)^{j+1}}
L(f,j+1)X^{\kappa-2-j}.
\]
Thus relations satisfied by the period polynomial translate into relations among the critical $L$-values.

\section{A three-component theta system and its period polynomials}\label{three-theta}

Consider the theta products
\[
F_1(\tau)=\theta_2(\tau)\theta_3(\tau)^8\theta_4(\tau),\qquad
F_2(\tau)=\theta_2(\tau)\theta_3(\tau)\theta_4(\tau)^8,
\]
and
\[
F_3(\tau)=\theta_2(\tau)^8\theta_3(\tau)\theta_4(\tau).
\]
The standard transformation formulas for the theta constants give
\begin{align*}
	F_1(-1/\tau)&=(-i\tau)^5F_1(\tau),\\
	F_2(-1/\tau)&=(-i\tau)^5F_3(\tau),\\
	F_3(-1/\tau)&=(-i\tau)^5F_2(\tau),
\end{align*}
and
\begin{align*}
	F_1(\tau+1)&=e^{\pi i/4}F_2(\tau),\\
	F_2(\tau+1)&=e^{\pi i/4}F_1(\tau),\\
	F_3(\tau+1)&=F_3(\tau).
\end{align*}

For $\gamma=\begin{pmatrix}a&b\\c&d\end{pmatrix}$, we use the weight $5$ slash operator
\[
(f|_5\gamma)(\tau):=(c\tau+d)^{-5}f\left(\frac{a\tau+b}{c\tau+d}\right).
\]
It follows that
\begin{align*}
	F_1|_5S&=-iF_1,& F_1|_5T&=e^{\pi i/4}F_2,\\
	F_2|_5S&=-iF_3,& F_2|_5T&=e^{\pi i/4}F_1,\\
	F_3|_5S&=-iF_2,& F_3|_5T&=F_3.
\end{align*}
Hence the space spanned by $F_1,F_2,F_3$ is stable under the weight $5$ slash actions of both $S$ and $T$.

Equivalently, writing
\[
\mathbf F(\tau):=\bigl(F_1(\tau),F_2(\tau),F_3(\tau)\bigr)^T,
\]
we have
\[
\mathbf F|_5S=-i
\begin{pmatrix}
	1&0&0\\
	0&0&1\\
	0&1&0
\end{pmatrix}\mathbf F,
\qquad
\mathbf F|_5T=
\begin{pmatrix}
	0&e^{\pi i/4}&0\\
	e^{\pi i/4}&0&0\\
	0&0&1
\end{pmatrix}\mathbf F.
\]
These transformation laws will be used below to derive the corresponding relations among the period integrals.

For $j=1,2,3$, define
	\[
	R_j(X):=\int_0^{i\infty}F_j(\tau)(\tau-X)^3\,d\tau.
	\]
	The transformation laws obtained above impose strong restrictions on these period polynomials.  In particular, the $S$-transformation of $F_1$, together with a modular-symbol relation arising from the ideal triangle with vertices $0$, $i\infty$, and $-1$, determines $R_1$ up to a scalar.
	
	\begin{theorem}
		The period polynomial of $F_1$ is given by
		\[
		R_1(X)=R_1(0)(1-X^2)(1+iX).
		\]
	\end{theorem}
	
	\begin{proof}
	Making the substitution $\tau=-1/z$ in the defining integral, and using
	$F_1(-1/z)=(-iz)^5F_1(z)$, we obtain
	\begin{align*}
		R_1(X)
		&=\int_{i\infty}^{0}(-iz)^5F_1(z)
		\left(-\frac{1}{z}-X\right)^3\frac{dz}{z^2}\\
		&=-i\int_0^{i\infty}F_1(z)(1+Xz)^3\,dz\\
		&=-iX^3R_1(-1/X).
	\end{align*}
Write
		\[
		R_1(X)=a+bX+cX^2+dX^3.
		\]
		Comparing coefficients in the preceding relation gives
		\[
		d=-ia,\qquad c=ib,
		\]
		and hence
		\[
		R_1(X)=a+bX+ibX^2-iaX^3.
		\]
It remains to determine the ratio $b/a$.  For this we apply Cauchy's theorem to the differential
\[
F_3(\tau)(\tau-X)^3\,d\tau
\]
around the positively oriented ideal triangle with vertices $0$, $\infty$, and $-1$,
\[
0\longrightarrow i\infty\longrightarrow -1\longrightarrow 0.
\]
We denote by $\mathcal{C}$ the upper semicircular geodesic in $\HH$ joining $-1$ to $0$, which forms the third side of the triangle.  More precisely, one truncates the triangle near its three cusps, applies Cauchy's theorem to the resulting compact region, and then passes to the limit.  The contributions from the auxiliary boundary arcs vanish by the cuspidal decay of the theta products.
		
		The integral along the first side is $R_3(X)$.  On the second side, from $i\infty$ to $-1$, make the substitution $\tau=z-1$.  Since
		\[
		F_3|_5T=F_3,
		\]
		this contribution is
		\[
		-R_3(X+1).
		\]
		
		For the third side $\mathcal{C}$, use the substitution
		\[
		\tau=-\frac{1}{z+1}=(ST)z.
		\]
		The transformation laws of $F_1,F_2,F_3$ give
		\[
		F_3|_5ST=(F_3|_5S)|_5T=-i(F_2|_5T)=-ie^{\pi i/4}F_1.
		\]
		Therefore
		\[
		F_3\left(-\frac{1}{z+1}\right)=-ie^{\pi i/4}(z+1)^5F_1(z).
		\]
		Since
		\[
		d\tau=\frac{dz}{(z+1)^2},
		\]

		the integral along the third side $\mathcal{C}$ becomes

		\begin{align*}
			\int_\mathcal{C}F_3(\tau)(\tau-X)^3\,d\tau
			&=\int_0^{i\infty}
			F_3\left(-\frac{1}{z+1}\right)
			\left(-\frac{1}{z+1}-X\right)^3
			\frac{dz}{(z+1)^2}\\
		&=ie^{\pi i/4}\int_0^{i\infty}F_1(z)\bigl(1+X(z+1)\bigr)^3\,dz\\
			&=ie^{\pi i/4}X^3\int_0^{i\infty}F_1(z)
			\left(z+1+\frac{1}{X}\right)^3\,dz\\
			&=ie^{\pi i/4}X^3R_1\left(-1-\frac{1}{X}\right).
		\end{align*}
		Cauchy's theorem therefore gives the relation
		\[
		R_3(X)-R_3(X+1)+ie^{\pi i/4}X^3R_1\left(-1-\frac{1}{X}\right)=0.
		\]
		
		Now $R_3(X)-R_3(X+1)$ has degree at most $2$.  Hence the coefficient of $X^3$ in the last term must vanish.  That coefficient is
		\[
		ie^{\pi i/4}R_1(-1),
		\]
		so
		\[
		R_1(-1)=0.
		\]
		Substituting the expression already obtained for $R_1$ gives
		\[
		0=(1+i)a+(-1+i)b,
		\]
		and therefore
		\[
		b=ia.
		\]
		It follows that
		\[
		R_1(X)=a(1+iX-X^2-iX^3)=a(1-X^2)(1+iX).
		\]
		Since $a=R_1(0)$, the result follows.
	\end{proof}

We now turn to the remaining two components of the theta system.  Recall that
	\[
	R_j(X):=\int_0^{i\infty}F_j(\tau)(\tau-X)^3\,d\tau,\qquad j=1,2,3,
	\]
	and that
	\[
	R_1(X)=R_1(0)(1-X^2)(1+iX).
	\]
	The $S$-transformation laws
	\[
	F_2|_5S=-iF_3,\qquad F_3|_5S=-iF_2
	\]
	give, by the substitution $\tau=-1/z$ in the defining period integrals,
	\[
	R_2(X)=-iX^3R_3\left(-\frac{1}{X}\right),\qquad
	R_3(X)=-iX^3R_2\left(-\frac{1}{X}\right).
	\]
	
	We also obtain a Manin relation by integrating the differential
	\[
	F_2(\tau)(\tau-X)^3\,d\tau
	\]
	around the positively oriented ideal triangle
	\[
	0\longrightarrow i\infty\longrightarrow -1\longrightarrow 0.
	\]
	The first side contributes $R_2(X)$.  For the second side, the
	substitution $\tau=z-1$ together with $F_1|_5T=\zeta F_2$ gives
	$-\zeta^{-1}R_1(X+1)$.  For the third side, we use
	\[
	\tau=-\frac{1}{z+1}=(ST)z
	\]
	together with $F_2|_5ST=-iF_3$.  Hence Cauchy's theorem gives
	\[
	R_2(X)-\zeta^{-1}R_1(X+1)
	+iX^3R_3\left(-1-\frac{1}{X}\right)=0,
	\qquad
	\zeta=e^{\pi i/4}.
	\]
	
	\begin{theorem}
		The period polynomials $R_2$ and $R_3$ satisfy
		\[
		R_2(0)=\frac{2\sqrt{2}}{3}R_1(0),
		\]
		and
		\[
		R_2(X)=R_2(0)\left(1+\frac{3i}{4}X-\frac{1}{4}X^2\right)-iR_3(0)X^3,
		\]
		\[
		R_3(X)=R_3(0)+R_2(0)\left(\frac{i}{4}X-\frac{3}{4}X^2-iX^3\right).
		\]
	\end{theorem}
	
\begin{proof}
	Put
	\[
	A:=R_1(0),\qquad \zeta:=e^{\pi i/4}.
	\]
	We already know that
	\[
	R_1(X)=A(1-X^2)(1+iX).
	\]
	Write
	\[
	R_2(X)=c_0+c_1X+c_2X^2+c_3X^3.
	\]
	The $S$-relation
	\[
	R_3(X)=-iX^3R_2\left(-\frac{1}{X}\right)
	\]
	then gives
	\[
	R_3(X)=ic_3-ic_2X+ic_1X^2-ic_0X^3.
	\]
	
	We now use the Manin relation in the form
	\[
	R_2(X)-\zeta^{-1}R_1(X+1)+iX^3R_3\left(-1-\frac{1}{X}\right)=0.
	\]
	We expand the last two terms explicitly.
	
	First,
	\[
	R_1(X+1)=-A\left((2+2i)X+(1+3i)X^2+iX^3\right).
	\]
	Since $\zeta^{-1}=(1-i)/\sqrt{2}$, it follows that
	\[
	-\zeta^{-1}R_1(X+1)=2\sqrt{2}AX+\sqrt{2}(2+i)AX^2+\frac{1+i}{\sqrt{2}}AX^3.
	\]
	
	Next, put
	\[
	Y=-1-\frac{1}{X}=-\frac{1+X}{X}.
	\]
	Using the expression for $R_3$, we obtain
	\begin{align*}
		iX^3R_3(Y)
		&=-c_3X^3-c_2X^2(1+X)-c_1X(1+X)^2-c_0(1+X)^3\\
		&=-c_0-(3c_0+c_1)X-(3c_0+2c_1+c_2)X^2\\
		&\qquad -(c_0+c_1+c_2+c_3)X^3.
	\end{align*}
	Adding $R_2(X)$ therefore gives
	\[
	R_2(X)+iX^3R_3\left(-1-\frac{1}{X}\right)
	=-3c_0X-(3c_0+2c_1)X^2-(c_0+c_1+c_2)X^3.
	\]
	In particular, the coefficient $c_3$ cancels from the Manin relation.
	
	Substituting these expressions into the Manin relation gives
	\[
	\left(2\sqrt{2}A-3c_0\right)X
	+\left(\sqrt{2}(2+i)A-3c_0-2c_1\right)X^2
	+\left(\frac{1+i}{\sqrt{2}}A-c_0-c_1-c_2\right)X^3=0.
	\]
	Comparing coefficients of $X$, $X^2$, and $X^3$, respectively, yields
	\[
	3c_0=2\sqrt{2}A,
	\]
	\[
	3c_0+2c_1=\sqrt{2}(2+i)A,
	\]
	and
	\[
	c_0+c_1+c_2=\frac{1+i}{\sqrt{2}}A.
	\]
	
	The first equation gives
	\[
	c_0=\frac{2\sqrt{2}}{3}A.
	\]
	Substituting this into the second gives
	\[
	2c_1=\sqrt{2}iA,
	\]
	and hence
	\[
	c_1=\frac{3i}{4}c_0.
	\]
	Finally, the third equation gives
	\[
	c_2=-\frac{\sqrt{2}}{6}A=-\frac{1}{4}c_0.
	\]
	Thus
	\[
	R_2(X)=c_0\left(1+\frac{3i}{4}X-\frac{1}{4}X^2\right)+c_3X^3.
	\]
	
	The coefficient $c_3$ is not constrained by the Manin relation.  Since
	\[
	R_3(0)=ic_3,
	\]
	we have
	\[
	c_3=-iR_3(0).
	\]
	Since $c_0=R_2(0)$, we conclude that
	\[
	R_2(0)=\frac{2\sqrt{2}}{3}R_1(0)
	\]
	and
	\[
	R_2(X)=R_2(0)\left(1+\frac{3i}{4}X-\frac{1}{4}X^2\right)-iR_3(0)X^3.
	\]
	Finally, substituting this expression into the $S$-relation gives
	\[
	R_3(X)=R_3(0)+R_2(0)\left(\frac{i}{4}X-\frac{3}{4}X^2-iX^3\right).
	\]
	This proves the theorem.
\end{proof}

\subsection{Eta-product realizations and critical $L$-values}

We have now determined the period structure of the three-component theta system up to the two constants $R_1(0)$ and $R_3(0)$.  We next identify the three components with eta quotients and use an additional duplication relation to determine the remaining period ratio before passing to critical $L$-values.

Using the theta--eta identities recalled in Section~\ref{sec:preliminaries}, we have
\begin{align*}
	F_1(\tau)&=2\frac{\eta(\tau)^{38}}{\eta(\tau/2)^{14}\eta(2\tau)^{14}},\\
	F_2(\tau)&=2\frac{\eta(\tau/2)^{14}}{\eta(\tau)^4},\\
	F_3(\tau)&=256\frac{\eta(2\tau)^{14}}{\eta(\tau)^4}.
\end{align*}
Accordingly, define
\[
h(\tau):=\frac{\eta(8\tau)^{38}}{\eta(4\tau)^{14}\eta(16\tau)^{14}},\qquad
u(\tau):=\frac{\eta(4\tau)^{14}}{\eta(8\tau)^4},\qquad
v(\tau):=\frac{\eta(2\tau)^{14}}{\eta(\tau)^4}.
\]
Then
\[
F_1(8\tau)=2h(\tau),\qquad F_2(8\tau)=2u(\tau),\qquad F_3(\tau)=256v(\tau).
\]

We first complete the period structure of the system.  Let
\[
g(\tau):=\eta(\tau)^4\eta(2\tau)^2\eta(4\tau)^4.
\]
The eta-product identity used by Rogers, Wan, and Zucker \cite{MR3338042}, cited there from the collection of Somos \cite{SomosEta}, has an equivalent formulation in terms of the classical Jacobi duplication formulas.  We record this formulation because it interacts naturally with the period-polynomial relations obtained above.

Put
\[
M(\tau):=\eta(2\tau)^4\eta(4\tau)^4.
\]
Using the theta--eta identities, we obtain
\[
g(\tau)=M(\tau)\theta_4(2\tau)^2,\qquad u(\tau)=M(\tau)\theta_3(4\tau)^2,
\]
and
\[
4g(2\tau)=M(\tau)\theta_2(4\tau)^2,\qquad v(\tau)=M(\tau)\theta_3(2\tau)^2.
\]
The Jacobi duplication formulas
\[
\theta_4(2\tau)^2=\theta_3(4\tau)^2-\theta_2(4\tau)^2
\]
and
\[
\theta_3(2\tau)^2=\theta_3(4\tau)^2+\theta_2(4\tau)^2
\]
therefore give
\[
u(\tau)=g(\tau)+4g(2\tau),\qquad v(\tau)=g(\tau)+8g(2\tau).
\]

These identities have a particularly simple interpretation at the level of period polynomials.  For a weight-$5$ cusp form $f$, write
\[
r_f(X):=\int_0^{i\infty}f(\tau)(\tau-X)^3\,d\tau,
\]
and define the rescaling operator $V_2$ by
\[
(V_2f)(\tau):=f(2\tau).
\]
The change of variables $z=2\tau$ gives
\[
r_{V_2f}(X)=\int_0^{i\infty}f(2\tau)(\tau-X)^3\,d\tau=\frac{1}{16}r_f(2X).
\]
Consequently,
\[
r_u(X)=r_g(X)+\frac{1}{4}r_g(2X),\qquad r_v(X)=r_g(X)+\frac{1}{2}r_g(2X).
\]
Thus the two period polynomials $r_u$ and $r_v$ are controlled by the single period polynomial $r_g$ together with its rescaling $r_g(2X)$.

In particular,
\[
r_u(0)=\frac{5}{4}r_g(0),\qquad r_v(0)=\frac{3}{2}r_g(0),
\]
and hence
\[
r_v(0)=\frac{6}{5}r_u(0).
\]
Since
\[
F_2(8\tau)=2u(\tau),\qquad F_3(\tau)=256v(\tau),
\]
we have
\[
R_2(X)=8192\,r_u\left(\frac{X}{8}\right),\qquad R_3(X)=256\,r_v(X).
\]
It follows that
\[
R_3(0)=\frac{3}{80}R_2(0).
\]
Together with
\[
R_2(0)=\frac{2\sqrt{2}}{3}R_1(0),
\]
we obtain
\[
R_3(0)=\frac{\sqrt{2}}{40}R_1(0).
\]
Thus the three period polynomials are determined up to the single normalization $R_1(0)$.  More explicitly,
\[
R_1(X)=R_1(0)(1-X^2)(1+iX),
\]
\[
R_2(X)=R_2(0)\left(1+\frac{3i}{4}X-\frac{1}{4}X^2-\frac{3i}{80}X^3\right),
\]
and
\[
R_3(X)=R_2(0)\left(\frac{3}{80}+\frac{i}{4}X-\frac{3}{4}X^2-iX^3\right).
\]

We now translate this completed period structure into relations among critical $L$-values.  The Mellin transform gives
\[
R_1(X)=\frac{3072}{\pi^4}L(h,4)+\frac{768i}{\pi^3}L(h,3)X-\frac{96}{\pi^2}L(h,2)X^2-\frac{8i}{\pi}L(h,1)X^3,
\]
\[
R_2(X)=\frac{3072}{\pi^4}L(u,4)+\frac{768i}{\pi^3}L(u,3)X-\frac{96}{\pi^2}L(u,2)X^2-\frac{8i}{\pi}L(u,1)X^3,
\]
and
\[
R_3(X)=\frac{96}{\pi^4}L(v,4)+\frac{192i}{\pi^3}L(v,3)X-\frac{192}{\pi^2}L(v,2)X^2-\frac{128i}{\pi}L(v,1)X^3.
\]

The same rescaling relations also recover the connection with the critical values of $g$.  If
\[
r_g(X)=\sum_{j=0}^3 a_jX^j,
\]
then
\[
r_u(X)=\sum_{j=0}^3\left(1+2^{j-2}\right)a_jX^j,\qquad
r_v(X)=\sum_{j=0}^3\left(1+2^{j-1}\right)a_jX^j.
\]
Since the coefficient of $X^j$ corresponds to the critical value with $s=4-j$, it follows that, for $1\leq s\leq4$,
\[
L(u,s)=\left(1+2^{2-s}\right)L(g,s),\qquad
L(v,s)=\left(1+2^{3-s}\right)L(g,s).
\]
Thus the corresponding $L$-function identities arise coefficientwise from the rescaling relations for the period polynomials.

\begin{corollary}
	The critical values associated with $h$, $u$, and $v$ satisfy
	\[
	L(h,4)=\frac{\pi}{4}L(h,3)=\frac{\pi^2}{32}L(h,2)=\frac{\pi^3}{384}L(h,1),
	\]
	\[
	L(u,4)=\frac{\pi}{3}L(u,3)=\frac{\pi^2}{8}L(u,2)=\frac{5\pi^3}{72}L(u,1),
	\]
	and
	\[
	L(v,4)=\frac{3\pi}{10}L(v,3)=\frac{\pi^2}{10}L(v,2)=\frac{\pi^3}{20}L(v,1).
	\]
	Moreover,
	\[
	L(u,4)=\frac{2\sqrt{2}}{3}L(h,4),
	\]
	and
	\[
	L(v,4)=\frac{6}{5}L(u,4).
	\]
	The critical values of $g$ satisfy
	\[
	L(g,4)=\frac{2\pi}{5}L(g,3)=\frac{\pi^2}{5}L(g,2)=\frac{\pi^3}{6}L(g,1).
	\]
	In particular,
	\[
	L(h,4)=\frac{15}{8\sqrt{2}}L(g,4).
	\]
\end{corollary}

\begin{proof}
	Comparing the coefficients of
	\[
	R_1(X)=R_1(0)(1-X^2)(1+iX)
	\]
	with its Mellin expansion gives the first chain.
	
	For $R_2$, comparison with
	\[
	R_2(X)=R_2(0)\left(1+\frac{3i}{4}X-\frac{1}{4}X^2-\frac{3i}{80}X^3\right)
	\]
	gives
	\[
	L(u,4)=\frac{\pi}{3}L(u,3)=\frac{\pi^2}{8}L(u,2)=\frac{5\pi^3}{72}L(u,1).
	\]
	Similarly, comparison with
	\[
	R_3(X)=R_2(0)\left(\frac{3}{80}+\frac{i}{4}X-\frac{3}{4}X^2-iX^3\right)
	\]
	gives
	\[
	L(v,4)=\frac{3\pi}{10}L(v,3)=\frac{\pi^2}{10}L(v,2)=\frac{\pi^3}{20}L(v,1).
	\]
	
	The relation
	\[
	R_2(0)=\frac{2\sqrt{2}}{3}R_1(0)
	\]
	gives
	\[
	L(u,4)=\frac{2\sqrt{2}}{3}L(h,4),
	\]
	while
	\[
	R_3(0)=\frac{3}{80}R_2(0)
	\]
	gives
	\[
	L(v,4)=\frac{6}{5}L(u,4).
	\]
	
	Finally, substituting
	\[
	L(u,s)=\left(1+2^{2-s}\right)L(g,s)
	\]
	into the relations for $u$ gives
	\[
	\frac{5}{4}L(g,4)=\frac{\pi}{2}L(g,3)=\frac{\pi^2}{4}L(g,2)=\frac{5\pi^3}{24}L(g,1),
	\]
	and hence
	\[
	L(g,4)=\frac{2\pi}{5}L(g,3)=\frac{\pi^2}{5}L(g,2)=\frac{\pi^3}{6}L(g,1).
	\]
	The final relation follows from
	\[
	\frac{5}{4}L(g,4)=L(u,4)=\frac{2\sqrt{2}}{3}L(h,4).
	\]
\end{proof}

The $S$-relation between $R_2$ and $R_3$ is reflected in the complementary critical-value identities
\[
L(u,4)=\frac{\pi^3}{24}L(v,1),\qquad L(u,3)=\frac{\pi}{4}L(v,2),
\]
and
\[
L(v,3)=\frac{\pi}{2}L(u,2),\qquad L(v,4)=\frac{\pi^3}{12}L(u,1).
\]
Thus the eta-product realization preserves both the coupled $S$-symmetry of the original theta system and the additional rescaling structure arising from the Jacobi duplication identities.		

\subsection{An intrinsic period polynomial for $g$}

The preceding analysis determines the period polynomial of $g$ indirectly through the auxiliary forms $u$ and $v$.  We now give a complementary derivation directly from $g$ itself.

Recall that
\[
g(\tau)=\eta(\tau)^4\eta(2\tau)^2\eta(4\tau)^4.
\]
It is convenient to rescale and put
\[
G_0(\tau):=g\left(\frac{\tau}{2}\right)
=\eta(\tau/2)^4\eta(\tau)^2\eta(2\tau)^4.
\]
A Gaussian theta series representation of this form, going back to Glaisher \cite{glaisher1907representations}, is

\[
G_0(\tau)
=
\frac{1}{4}\sum_{m,n\in\mathbb Z}
(n-im)^4e^{\pi i\tau(n^2+m^2)}.
\]
Define the companion theta series
\[
G_1(\tau)
:=
\frac{1}{4}\sum_{m,n\in\mathbb Z}
(-1)^{m+n}(n-im)^4e^{\pi i\tau(n^2+m^2)}.
\]
Let
\[
P_j(X):=\int_0^{i\infty}G_j(\tau)(\tau-X)^3\,d\tau,
\qquad j=0,1,
\]
and write
\[
P_g(X):=P_0(X).
\]
Since $G_0(\tau)=g(\tau/2)$, the change of variables $\tau=2z$ gives
\[
P_g(X)=16r_g\left(\frac{X}{2}\right).
\]
Thus determining $P_g$ is equivalent to determining the intrinsic period polynomial $r_g$.

\begin{theorem}
	The period polynomial associated with $G_0(\tau)=g(\tau/2)$ is
	\[
	P_g(X)=\frac{P_g(0)}{2}\left(2+5iX-5X^2-2iX^3\right).
	\]
	Equivalently,
	\[
	r_g(X)=r_g(0)\left(1+5iX-10X^2-8iX^3\right).
	\]
\end{theorem}

\begin{proof}
	We first determine the action of $S$ on $G_0$ directly from the transformation law of the Dedekind eta function.  Since
	\[
	G_0(\tau)=\eta(\tau/2)^4\eta(\tau)^2\eta(2\tau)^4,
	\]
	we have
	\[
	G_0\left(-\frac{1}{\tau}\right)
	=
	\eta\left(-\frac{1}{2\tau}\right)^4
	\eta\left(-\frac{1}{\tau}\right)^2
	\eta\left(-\frac{2}{\tau}\right)^4.
	\]
	Using
	\[
	\eta\left(-\frac{1}{\tau}\right)=\sqrt{-i\tau}\,\eta(\tau),
	\]
	in the three forms
	\[
	\eta\left(-\frac{1}{2\tau}\right)=\sqrt{-2i\tau}\,\eta(2\tau),
	\qquad
	\eta\left(-\frac{1}{\tau}\right)=\sqrt{-i\tau}\,\eta(\tau),
	\]
	and
	\[
	\eta\left(-\frac{2}{\tau}\right)
	=
	\sqrt{-\frac{i\tau}{2}}\,\eta\left(\frac{\tau}{2}\right),
	\]
	we obtain
	\[
	G_0\left(-\frac{1}{\tau}\right)
	=
	(-2i\tau)^2(-i\tau)\left(-\frac{i\tau}{2}\right)^2
	G_0(\tau)
	=
	(-i\tau)^5G_0(\tau).
	\]
	Thus
	\[
	G_0|_5S=-iG_0.
	\]
	
	Substituting $\tau=-1/z$ in the period integral therefore gives
	\[
	P_0(X)=-iX^3P_0\left(-\frac{1}{X}\right).
	\]
	Write
	\[
	P_0(X)=a+bX+cX^2+dX^3.
	\]
	Comparison of coefficients gives
	\[
	c=ib,\qquad d=-ia,
	\]
	and hence
	\[
	P_0(X)=a+bX+ibX^2-iaX^3.
	\]
	
	We next use the Gaussian theta series representation.  Since
	\[
	e^{\pi i(n^2+m^2)}=(-1)^{n^2+m^2}=(-1)^{n+m},
	\]
	translation by $1$ gives
	\[
	G_0(\tau+1)=G_1(\tau).
	\]
	Similarly,
	\[
	G_1(\tau+1)=G_0(\tau).
	\]
	
The Gaussian theta-series representation also yields
\[
G_0(\tau)+G_1(\tau)=-8G_0(2\tau).
\]
Indeed, $1+(-1)^{m+n}$ restricts the Gaussian sum to the index two
sublattice $(1+i)\mathbb Z[i]\subset\mathbb Z[i]$; under the corresponding
change of variables the norm is multiplied by $2$ and the quartic factor by
$(1-i)^4=-4$, giving the stated identity.

%
Passing to period polynomials therefore gives
\[
P_1(X)=-P_0(X)-\frac12P_0(2X).
\]
	Since
	\[
	P_0(X)=a+bX+ibX^2-iaX^3,
	\]
	this gives
	\[
	P_1(X)
	=
	-\frac{3}{2}a-2bX-3ibX^2+5iaX^3.
	\]
	
	It remains to determine the ratio $b/a$.  We apply the same ideal-triangle argument used above, integrating
	\[
	G_0(\tau)(\tau-X)^3\,d\tau
	\]
	around the positively oriented triangle with vertices $0$, $i\infty$, and $-1$.
	
	The side from $0$ to $i\infty$ contributes
	\[
	P_0(X).
	\]
	On the side from $i\infty$ to $-1$, put $\tau=z-1$.  Since
	\[
	G_0(z-1)=G_1(z),
	\]
	this side contributes
	\[
	-P_1(X+1).
	\]
	For the third side, put
	\[
	\tau=-\frac{1}{z+1}.
	\]
	Using the $S$- and $T$-relations,
	\[
	G_0\left(-\frac{1}{z+1}\right)
	=
	(-i(z+1))^5G_0(z+1)
	=
	-i(z+1)^5G_1(z).
	\]
	Therefore the third side contributes
	\[
	iX^3P_1\left(-1-\frac{1}{X}\right).
	\]
	Cauchy's theorem gives the Manin relation
	\[
	P_0(X)-P_1(X+1)
	+iX^3P_1\left(-1-\frac{1}{X}\right)=0.
	\]
	
Taking the constant term in the Manin relation gives
\[
a+\left(\frac{3}{2}a+2b+3ib-5ia\right)+5a=0,
\]
and hence
\[
b=\frac{5i}{2}a.
\]
Therefore
\[
P_0(X)=a\left(1+\frac{5i}{2}X-\frac{5}{2}X^2-iX^3\right).
\]
Since $a=P_0(0)=P_g(0)$, we obtain
\[
P_g(X)=\frac{P_g(0)}{2}\left(2+5iX-5X^2-2iX^3\right).
\]
	
	Finally,
	\[
	P_g(X)=16r_g\left(\frac{X}{2}\right),
	\]
	so
	\[
	r_g(X)=\frac{1}{16}P_g(2X).
	\]
	Since $P_g(0)=16r_g(0)$, we obtain
	\[
	r_g(X)
	=
	r_g(0)\left(1+5iX-10X^2-8iX^3\right).
	\]
\end{proof}

The preceding subsection already translated this period polynomial into the corresponding relations among the critical values of $g$.  The present argument gives an intrinsic derivation: the $S$ and $T$ symmetries determine the period polynomial up to two coefficients, while the parity decomposition of the Gaussian theta series supplies the additional level $2$ relation that fixes their ratio.

We next translate the completed period relations into identities among moments of complete elliptic integrals.
\subsection{Elliptic integral consequences}

We now translate the critical-value relations obtained above into identities among moments of complete elliptic integrals.  We first recall the elliptic-integral representations of Rogers, Wan, and Zucker \cite{MR3338042} for the critical values of
\[
g(\tau)=\eta(\tau)^4\eta(2\tau)^2\eta(4\tau)^4:
\]
\[
30L(g,4)=\int_0^1 K'(k)^3\,dk,
\]
\[
2\pi L(g,3)=\int_0^1 kK(k)K'(k)^2\,dk,
\]
\[
\pi^2L(g,2)=\int_0^1 kK(k)^2K'(k)\,dk,
\]
and
\[
\pi^3L(g,1)=\int_0^1 kK(k)^3\,dk.
\]
Hence the critical-value relation
\[
L(g,4)=\frac{2\pi}{5}L(g,3)=\frac{\pi^2}{5}L(g,2)=\frac{\pi^3}{6}L(g,1)
\]
recovers the known cubic moment identity
\[
\int_0^1 K'(k)^3\,dk
=6\int_0^1 kK(k)K'(k)^2\,dk
=6\int_0^1 kK(k)^2K'(k)\,dk
=5\int_0^1 kK(k)^3\,dk
\]
studied by Wan \cite{MR2845511} and by Rogers, Wan, and Zucker \cite{MR3338042}. An analytic proof of the corresponding cubic-moment relation was later given by Zhou \cite{MR3231318}.

The critical values associated with $u$ and $v$ admit analogous direct elliptic integral representations.  Put
\[
t=\frac{K'(k)}{K(k)},\qquad \tau=it.
\]
Then
\[
\frac{dt}{dk}=-\frac{\pi}{2k(1-k^2)K(k)^2},
\]
and the theta parametrization gives
\[
F_2(it)=\left(\frac{2K(k)}{\pi}\right)^5\sqrt{k}(1-k^2)^2,
\]
and
\[
F_3(it)=\left(\frac{2K(k)}{\pi}\right)^5k^4(1-k^2)^{1/4}.
\]

	Combining these formulas with the Mellin transform yields
	\begin{align*}
		192L(u,4)&=\int_0^1\frac{1-k^2}{\sqrt{k}}K'(k)^3\,dk,\\
		16\pi L(u,3)&=\int_0^1\frac{1-k^2}{\sqrt{k}}K(k)K'(k)^2\,dk,\\
		2\pi^2L(u,2)&=\int_0^1\frac{1-k^2}{\sqrt{k}}K(k)^2K'(k)\,dk,\\
		\frac{\pi^3}{2}L(u,1)&=\int_0^1\frac{1-k^2}{\sqrt{k}}K(k)^3\,dk,
	\end{align*}
	and
	\begin{align*}
		6L(v,4)&=\int_0^1\frac{k^3}{(1-k^2)^{3/4}}K'(k)^3\,dk,\\
		4\pi L(v,3)&=\int_0^1\frac{k^3}{(1-k^2)^{3/4}}K(k)K'(k)^2\,dk,\\
		4\pi^2L(v,2)&=\int_0^1\frac{k^3}{(1-k^2)^{3/4}}K(k)^2K'(k)\,dk,\\
		8\pi^3L(v,1)&=\int_0^1\frac{k^3}{(1-k^2)^{3/4}}K(k)^3\,dk.
	\end{align*}
	
Using the complete critical-value relations obtained above, these representations give
\[
\begin{aligned}
	\int_0^1 \frac{1-k^2}{\sqrt{k}}K'(k)^3\,dk
	&=4\int_0^1 \frac{1-k^2}{\sqrt{k}}K(k)K'(k)^2\,dk\\
	&=12\int_0^1 \frac{1-k^2}{\sqrt{k}}K(k)^2K'(k)\,dk\\
	&=\frac{80}{3}\int_0^1 \frac{1-k^2}{\sqrt{k}}K(k)^3\,dk,
\end{aligned}
\]
and
\[
\begin{aligned}
	\int_0^1 \frac{k^3}{(1-k^2)^{3/4}}K(k)^3\,dk
	&=4\int_0^1 \frac{k^3}{(1-k^2)^{3/4}}K(k)^2K'(k)\,dk\\
	&=12\int_0^1 \frac{k^3}{(1-k^2)^{3/4}}K(k)K'(k)^2\,dk\\
	&=\frac{80}{3}\int_0^1 \frac{k^3}{(1-k^2)^{3/4}}K'(k)^3\,dk.
\end{aligned}
\]
The complementary-modulus substitution $k\mapsto \sqrt{1-k^2}$ interchanges
the two families, reflecting the $S$-symmetry exchanging the $F_2$ and $F_3$
components.

Taken together, these identities show that the three-component theta system provides a common framework for the period relations associated with $h$ and $g$ and for the corresponding families of elliptic integral moment identities.

\section{A pair of weight-$4$ eta quotients}\label{weight-4}
	
	Rogers, Wan, and Zucker \cite{MR3338042} considered the two weight-$4$ cusp forms
	\[
	f_1(\tau):=\frac{\eta(4\tau)^{16}}{\eta(2\tau)^4\eta(8\tau)^4},
	\qquad
	f_2(\tau):=\eta(2\tau)^4\eta(4\tau)^4,
	\]
	and obtained relations between their critical $L$-values.  We give a period-polynomial and modular-symbol explanation of these relations.
	
	Set
	\[
	F(\tau):=f_1(\tau),\qquad G(\tau):=f_2(2\tau),
	\]
	and let
	\[
	\sigma=
	\begin{pmatrix}
		1&0\\
		4&1
	\end{pmatrix}.
	\]
	Thus
	\[
	\sigma\tau=\frac{\tau}{4\tau+1},
	\qquad
	\sigma(0)=0,\qquad
	\sigma(i\infty)=\frac14.
	\]
	
	We shall also use the elementary eta identity
	\[
	\eta\left(\tau+\frac12\right)
	=
	e^{\pi i/24}\frac{\eta(2\tau)^3}{\eta(\tau)\eta(4\tau)},
	\]
	which follows directly from the infinite product for $\eta$.
	
	\begin{proposition}
		Under the weight-$4$ slash operator,
		\[
		F|_4\sigma=-4iG,
		\qquad
		G|_4\sigma=-\frac{i}{4}F.
		\]
	\end{proposition}
	
	\begin{proof}
		Put
		\[
		w=4\tau+1.
		\]
		Then
		\[
		4\sigma\tau=1-\frac1w,\qquad
		2\sigma\tau=\frac12-\frac{1}{2w},\qquad
		8\sigma\tau=2-\frac{2}{w}.
		\]
		Using
		\[
		\eta\left(-\frac1\tau\right)=\sqrt{-i\tau}\,\eta(\tau),
		\qquad
		\eta(\tau+1)=e^{\pi i/12}\eta(\tau),
		\]
		together with the half-translation identity above, we obtain
		\[
		\eta(4\sigma\tau)
		=
		e^{\pi i/12}\sqrt{-iw}\,\eta(w),
		\]
		\[
		\eta(8\sigma\tau)
		=
		e^{\pi i/6}\sqrt{-\frac{iw}{2}}\,\eta\left(\frac w2\right),
		\]
		and
		\[
		\eta(2\sigma\tau)
		=
		e^{\pi i/24}\sqrt{-iw}\,
		\frac{\eta(w)^3}{\eta(w/2)\eta(2w)}.
		\]
		Substituting these expressions into $F(\sigma\tau)$ gives
		\[
		F(\sigma\tau)
		=
		4iw^4\eta(w)^4\eta(2w)^4.
		\]
		Since
		\[
		\eta(w)^4\eta(2w)^4
		=
		-\eta(4\tau)^4\eta(8\tau)^4
		=
		-G(\tau),
		\]
		we obtain
		\[
		F(\sigma\tau)=-4iw^4G(\tau),
		\]
		and hence
		\[
		F|_4\sigma=-4iG.
		\]
		
		Similarly,
		\[
		G(\sigma\tau)
		=
		-\frac{w^4}{4}\eta(w)^4\eta\left(\frac w2\right)^4.
		\]
		Now
		\[
		\eta\left(\frac w2\right)
		=
		\eta\left(2\tau+\frac12\right)
		=
		e^{\pi i/24}
		\frac{\eta(4\tau)^3}{\eta(2\tau)\eta(8\tau)},
		\]
		and therefore
		\[
		\eta(w)^4\eta\left(\frac w2\right)^4
		=
		i\frac{\eta(4\tau)^{16}}{\eta(2\tau)^4\eta(8\tau)^4}
		=
		iF(\tau).
		\]
		Thus
		\[
		G(\sigma\tau)
		=
		-\frac{i}{4}w^4F(\tau),
		\]
		which gives
		\[
		G|_4\sigma=-\frac{i}{4}F.
		\]
	\end{proof}
	
	Define the period polynomials
	\[
	R_j(X):=\int_0^{i\infty}f_j(\tau)(\tau-X)^2\,d\tau,
	\qquad j=1,2.
	\]
	Before using the modular-symbol decomposition, we record the reflection relations coming from the Fricke involutions.  Directly from the transformation law of $\eta$,
	\[
	f_1\left(-\frac{1}{16\tau}\right)=256\tau^4f_1(\tau),
	\qquad
	f_2\left(-\frac{1}{8\tau}\right)=64\tau^4f_2(\tau).
	\]

	Substitution in the defining period integrals gives
	\begin{align}
		R_1(X)&=-16X^2R_1\left(-\frac{1}{16X}\right),
		\label{eq:R1-Fricke}\\
		R_2(X)&=-8X^2R_2\left(-\frac{1}{8X}\right).
		\label{eq:R2-Fricke}
	\end{align}
	These identities are the period-polynomial form of the functional equations induced by the corresponding Fricke involutions.  Indeed, if
	\[
	R_j(X)=A_j+B_jX+C_jX^2,
	\]
	then comparison of constant and quadratic coefficients gives
	\[
	A_1=-\frac{C_1}{16},
	\qquad
	A_2=-\frac{C_2}{8}.
	\]
Expanding the period polynomial and evaluating the resulting period integrals by the Mellin transform, we obtain
\begin{equation}\label{eq:mellin-Rj}
	R_j(X)=-\frac{i}{4\pi^3}L(f_j,3)
	+\frac{1}{2\pi^2}L(f_j,2)X+\frac{i}{2\pi}L(f_j,1)X^2.
\end{equation}
Comparing coefficients gives
\[
A_j=-\frac{i}{4\pi^3}L(f_j,3),
\qquad
C_j=\frac{i}{2\pi}L(f_j,1).
\]
	It follows immediately that
	\[
	L(f_1,3)=\frac{\pi^2}{8}L(f_1,1),
	\qquad
	L(f_2,3)=\frac{\pi^2}{4}L(f_2,1).
	\]
	Thus the endpoint critical values for each form are already related by its Fricke symmetry.

	We now use the modular-symbol decomposition
	\[
	\{0,\infty\}
	=
	\left\{0,\frac14\right\}
	+
	\left\{\frac14,\infty\right\}.
	\]
	Accordingly,
	\[
	R_1(X)
	=
	\int_0^{1/4}F(\tau)(\tau-X)^2\,d\tau
	+
	\int_{1/4}^{i\infty}F(\tau)(\tau-X)^2\,d\tau.
	\]
	
	For the first integral, put
	\[
	\tau=\sigma z=\frac{z}{4z+1}.
	\]
	Then
	\[
	d\tau=\frac{dz}{(4z+1)^2},
	\]
	and
	\[
	\tau-X
	=
	\frac{(1-4X)z-X}{4z+1}.
	\]
	Since
	\[
	F(\sigma z)=-4i(4z+1)^4G(z),
	\]
	we obtain
	\[
	\begin{aligned}
		\int_0^{1/4}F(\tau)(\tau-X)^2\,d\tau
		&=
		-4i(1-4X)^2
		\int_0^{i\infty}
		G(z)
		\left(z-\frac{X}{1-4X}\right)^2\,dz.
	\end{aligned}
	\]
	Because $G(z)=f_2(2z)$, the change of variables $w=2z$ gives
	\[
	\int_0^{i\infty}G(z)(z-A)^2\,dz
	=
	\frac18R_2(2A).
	\]
	Hence
	\[
	\int_0^{1/4}F(\tau)(\tau-X)^2\,d\tau
	=
	-\frac{i}{2}(1-4X)^2
	R_2\left(\frac{2X}{1-4X}\right).
	\]
	
	For the second integral, put
	\[
	\tau=\frac{z+1}{4}.
	\]
	The half-translation identity for $\eta$ gives
	\[
	F\left(\frac{z+1}{4}\right)
	=
	if_2\left(\frac z4\right).
	\]
	Therefore
	\[
	\begin{aligned}
		\int_{1/4}^{i\infty}F(\tau)(\tau-X)^2\,d\tau
		&=
		iR_2\left(-\frac{1-4X}{4}\right).
	\end{aligned}
	\]
	Using the reflection relation
	\[
	R_2(Y)=-8Y^2R_2\left(-\frac{1}{8Y}\right)
	\]
	with
	\[
	Y=-\frac{1-4X}{4},
	\]
	we obtain
	\[
	\int_{1/4}^{i\infty}F(\tau)(\tau-X)^2\,d\tau
	=
	-\frac{i}{2}(1-4X)^2
	R_2\left(\frac{1}{2(1-4X)}\right).
	\]

Combining the two sides gives
\begin{equation}\label{eq:weight4-modular-symbol}
	R_1(X)
	=
	-\frac{i}{2}(1-4X)^2
	\left[
	R_2\left(\frac{2X}{1-4X}\right)
	+
	R_2\left(\frac{1}{2(1-4X)}\right)
	\right].
\end{equation}
Comparing coefficients and using \eqref{eq:mellin-Rj} yields the cross-relations among the critical values.
	
	\begin{theorem}
		The critical values of $f_1$ and $f_2$ satisfy
		\[
		L(f_1,3)=\frac{\pi}{2}L(f_2,2)
		=\frac{\pi^2}{8}L(f_1,1),
		\]
		and
		\[
		L(f_2,3)=\frac{\pi}{4}L(f_1,2)
		=\frac{\pi^2}{4}L(f_2,1).
		\]
	\end{theorem}

\begin{proof}
	The endpoint relations
	\[
	L(f_1,3)=\frac{\pi^2}{8}L(f_1,1),
	\qquad
	L(f_2,3)=\frac{\pi^2}{4}L(f_2,1)
	\]
	were obtained above from the Fricke involutions. It remains to obtain the two cross-relations.
	
	Recall that
	\[
	R_j(X)=A_j+B_jX+C_jX^2.
	\]
	Comparing the constant and quadratic coefficients in the reflection relation
	\eqref{eq:R2-Fricke} gives
	\[
	A_2=-\frac{C_2}{8}.
	\]
	
	We now extract the constant and linear terms from the modular-symbol identity
	\eqref{eq:weight4-modular-symbol}. Working modulo $X^2$, we have
	\[
	(1-4X)^2\equiv 1-8X \pmod{X^2},
	\]
	and
	\[
	\frac{2X}{1-4X}\equiv 2X,
	\qquad
	\frac{1}{2(1-4X)}\equiv \frac{1}{2}+2X
	\pmod{X^2}.
	\]
	Hence \eqref{eq:weight4-modular-symbol} becomes
	\[
	R_1(X)
	\equiv
	-\frac{i}{2}(1-8X)
	\left[
	R_2(2X)
	+
	R_2\left(\frac{1}{2}+2X\right)
	\right]
	\pmod{X^2}.
	\]
	Since
	\[
	R_2(2X)\equiv A_2+2B_2X\pmod{X^2},
	\]
	and
	\[
	R_2\left(\frac{1}{2}+2X\right)
	\equiv
	A_2+\frac{1}{2}B_2+\frac{1}{4}C_2
	+(2B_2+2C_2)X
	\pmod{X^2},
	\]
	we obtain
	\[
	R_1(X)
	\equiv
	-\frac{i}{2}(1-8X)
	\left[
	2A_2+\frac{1}{2}B_2+\frac{1}{4}C_2
	+(4B_2+2C_2)X
	\right]
	\pmod{X^2}.
	\]
	Using $A_2=-C_2/8$, this simplifies to
	\[
	R_1(X)
	\equiv
	-\frac{i}{2}(1-8X)
	\left[
	\frac{1}{2}B_2+(4B_2-16A_2)X
	\right]
	\pmod{X^2},
	\]
	and hence
	\[
	R_1(X)
	\equiv
	-\frac{i}{4}B_2+8iA_2X
	\pmod{X^2}.
	\]
	Comparing the constant and linear coefficients therefore gives
	\[
	A_1=-\frac{i}{4}B_2,
	\qquad
	B_1=8iA_2.
	\]
	
	By \eqref{eq:mellin-Rj},
	\[
	A_1=-\frac{i}{4\pi^3}L(f_1,3),
	\qquad
	B_2=\frac{1}{2\pi^2}L(f_2,2),
	\]
	so the first relation gives
	\[
	L(f_1,3)=\frac{\pi}{2}L(f_2,2).
	\]
	Similarly,
	\[
	B_1=\frac{1}{2\pi^2}L(f_1,2),
	\qquad
	A_2=-\frac{i}{4\pi^3}L(f_2,3),
	\]
	and the second relation gives
	\[
	L(f_2,3)=\frac{\pi}{4}L(f_1,2).
	\]
	Combining these cross-relations with the endpoint relations proves the result.
\end{proof}
	
	Thus the relations between the critical values of $f_1$ and $f_2$ arise from two complementary modular symmetries: the Fricke involutions relate the endpoint periods of each form separately, while the modular-symbol decomposition through the cusp $1/4$ relates the periods of the two forms to one another.
\section{A quadratic twist of conductor $3$}\label{conductor3}
\label{sec:conductor-three}

The preceding weight-$4$ example suggests that a quadratic twist of small conductor may give rise to a small slash-stable system. We illustrate this mechanism with the nontrivial Dirichlet character modulo $3$.

Let
\[
\chi_{-3}(n):=\left(\frac{-3}{n}\right)
\]
and consider the weight-$4$ cusp form
\[
\phi(\tau):=\eta(\tau)^2\eta(2\tau)^2\eta(3\tau)^2\eta(6\tau)^2.
\]
The form $\phi$ is a newform of level $6$.  Let
\[
\psi:=\phi\otimes\chi_{-3}.
\]
Then $\psi$ is a weight-$4$ newform of level $18$.  Since $\chi_{-3}(n)=0$ for $3\mid n$, the Fourier expansion of $\psi$ contains no terms $q^{3n}$.

The additive description of the twist becomes particularly simple in this case.  For $3\nmid n$,
\[
e^{-2\pi i n/3}=-\frac{1}{2}-\frac{i\sqrt{3}}{2}\chi_{-3}(n).
\]
Moreover, since $\phi$ is a $U_3$-eigenform with eigenvalue $-3$, the $3$-depletion of $\phi$ is
\[
H(\tau):=\phi(\tau)+3\phi(3\tau)=\psi\otimes\chi_{-3}.
\]
It follows that
\begin{equation}\label{eq:cond3-translation}
\psi\left(\tau-\frac{1}{3}\right)=	-\frac{1}{2}\psi(\tau)-\frac{i\sqrt{3}}{2}H(\tau).
\end{equation}
Thus translation by $-1/3$ already acts on the two-dimensional space generated by $\psi$ and $H$.

For the calculations below, we extend the weight-$4$ slash operator to matrices
\[
\gamma=
\begin{pmatrix}
	a&b\\
	c&d
\end{pmatrix}
\in\operatorname{GL}_2^+(\mathbb{Q})
\]
by
\[
(F|_4\gamma)(\tau)
:=
\det(\gamma)^2(c\tau+d)^{-4}
F\left(\frac{a\tau+b}{c\tau+d}\right).
\]
Let
\[
W_{18}:=
\begin{pmatrix}
	0&-1\\
	18&0
\end{pmatrix}.
\]
The Fricke involutions of $\phi$ and $\psi$ have eigenvalue $+1$.  At the common level $18$ this gives

\begin{equation}\label{eq:cond3-fricke-basic}
	\psi|_4W_{18}=\psi,\qquad
	\phi|_4W_{18}=9\phi(3\tau),\qquad
	\phi(3\tau)|_4W_{18}=\frac{1}{9}\phi.
\end{equation}

It is convenient to introduce the second oldform combination
\[
J(\tau):=\phi(\tau)+27\phi(3\tau).
\]
Then
\begin{equation}\label{eq:cond3-fricke-HJ}
	H|_4W_{18}=\frac{1}{3}J,
	\qquad
	J|_4W_{18}=3H.
\end{equation}

Set
\[
U:=
\begin{pmatrix}
	1&-\frac{1}{3}\\
	0&1
\end{pmatrix},
\qquad
\sigma:=W_{18}^{-1}UW_{18}
=
\begin{pmatrix}
	1&0\\
	6&1
\end{pmatrix}.
\]
Thus
\[
\sigma(i\infty)=\frac{1}{6}.
\]
Combining \eqref{eq:cond3-translation} with \eqref{eq:cond3-fricke-HJ} gives
\begin{equation}\label{eq:cond3-sigma}
	\psi|_4\sigma
	=
	-\frac{1}{2}\psi-\frac{i\sqrt{3}}{6}J.
\end{equation}
In particular, the conductor-$3$ translation, after Fricke conjugation, produces a small slash-stable system associated with the cusp $1/6$.

To obtain a closed cusp relation, define
\[
\mathcal A:=\sigma W_{18}
=
\begin{pmatrix}
	0&-1\\
	18&-6
\end{pmatrix}.
\]
A direct calculation gives
\[
\mathcal A^4=-324I,
\]
so that $\mathcal A$ has projective order $4$.  Its action on the cusps is
\[
i\infty\longmapsto0\longmapsto\frac{1}{6}
\longmapsto\frac{1}{3}\longmapsto i\infty.
\]
Thus the orbit of the edge from $i\infty$ to $0$ forms the boundary of the ideal quadrilateral with vertices
\[
i\infty,\qquad 0,\qquad \frac{1}{6},\qquad \frac{1}{3}.
\]

Since
\[
\mathcal A=\sigma W_{18}=W_{18}U,
\]
the action of $\mathcal A$ may be computed from the translation and Fricke
relations above. In addition to \eqref{eq:cond3-translation}, we have
\[
H|_4U=-\frac12H-\frac{i\sqrt3}{2}\psi,
\]
and, since $J=H+24\phi(3\tau)$ and $\phi(3\tau)$ is invariant under
$\tau\mapsto\tau-\frac13$,
\[
J|_4U=J-\frac32H-\frac{i\sqrt3}{2}\psi.
\]

Combining these identities with the Fricke relations
\eqref{eq:cond3-fricke-basic} and \eqref{eq:cond3-fricke-HJ} gives

\begin{align}
	\psi|_4\mathcal A
	&=
	-\frac{1}{2}\psi-\frac{i\sqrt{3}}{2}H,
	\label{eq:cond3-A1}\\
	\psi|_4\mathcal A^2
	&=
	\frac{i\sqrt{3}}{2}H-\frac{i\sqrt{3}}{6}J,
	\label{eq:cond3-A2}\\
	\psi|_4\mathcal A^3
	&=
	-\frac{1}{2}\psi+\frac{i\sqrt{3}}{6}J.
	\label{eq:cond3-A3}
\end{align}
Hence the entire cusp cycle closes in the three-dimensional space
\[
\operatorname{span}\{\psi,H,J\}.
\]

For a weight-$4$ cusp form $F$, write
\[
R_F(X):=\int_0^{i\infty}F(\tau)(\tau-X)^2\,d\tau.
\]
We first record the change-of-variables formula underlying the quadrilateral relation.  If
\[
\gamma=
\begin{pmatrix}
	a&b\\
	c&d
\end{pmatrix},
\qquad D=\det(\gamma)>0,
\]
then
\[
\int_{\gamma\alpha}^{\gamma\beta}
F(\tau)(\tau-X)^2\,d\tau
=
\frac{(a-cX)^2}{D}
\int_\alpha^\beta
(F|_4\gamma)(z)
\left(
z-\frac{dX-b}{a-cX}
\right)^2\,dz.
\]

Applying this identity to the four sides of the ideal quadrilateral and using
Cauchy's theorem gives

\begin{equation}\label{eq:cond3-period}
	\begin{aligned}
		0={}&R_\psi(X)
		+18X^2R_{\psi|_4\mathcal A}
		\left(\frac{1}{3}-\frac{1}{18X}\right)
		+(1-6X)^2R_{\psi|_4\mathcal A^2}
		\left(\frac{1-3X}{3(1-6X)}\right)\\
		&\quad
		+2(1-3X)^2R_{\psi|_4\mathcal A^3}
		\left(\frac{1}{6(1-3X)}\right).
	\end{aligned}
\end{equation}
The auxiliary forms $H$ and $J$ are oldform combinations of $\phi$.  Since
\[
R_{\phi(3\cdot)}(X)=\frac{1}{27}R_\phi(3X),
\]
we have
\begin{equation}\label{eq:cond3-oldforms}
	R_H(X)=R_\phi(X)+\frac{1}{9}R_\phi(3X),
	\qquad
	R_J(X)=R_\phi(X)+R_\phi(3X).
\end{equation}
Consequently, \eqref{eq:cond3-period} is an explicit relation between the period polynomials of the two primitive forms $\phi$ and $\psi$.

Write
\[
R_\psi(X)=A_\psi+B_\psi X+C_\psi X^2,
\qquad
R_\phi(X)=A_\phi+B_\phi X+C_\phi X^2.
\]
The Fricke involutions at levels $18$ and $6$ give
\[
R_\psi(X)=-18X^2R_\psi\left(-\frac{1}{18X}\right),
\qquad
R_\phi(X)=-6X^2R_\phi\left(-\frac{1}{6X}\right),
\]
and therefore
\begin{equation}\label{eq:cond3-fricke-coeff}
	C_\psi=-18A_\psi,
	\qquad
	C_\phi=-6A_\phi.
\end{equation}
Using \eqref{eq:cond3-oldforms} and \eqref{eq:cond3-fricke-coeff}, we obtain
\[
R_H(X)
=
\frac{10}{9}A_\phi+\frac{4}{3}B_\phi X-12A_\phi X^2,
\]
and
\[
R_J(X)
=
2A_\phi+4B_\phi X-60A_\phi X^2.
\]

We now extract only the constant and linear terms of the quadrilateral relation \eqref{eq:cond3-period}.  Substituting \eqref{eq:cond3-A1}--\eqref{eq:cond3-A3} and reducing modulo $X^2$ gives
\begin{align}
	0\equiv{}&
	A_\psi-\frac{1}{6}B_\psi
	+\frac{i\sqrt{3}}{9}\left(10A_\phi+2B_\phi\right) \notag\\
	&+
	\left(
	2B_\psi-\frac{40i\sqrt{3}}{3}A_\phi
	\right)X
	\pmod{X^2}.
	\label{eq:cond3-modX2}
\end{align}
Comparison of the linear coefficient gives
\[
B_\psi=\frac{20i\sqrt{3}}{3}A_\phi.
\]
Substitution into the constant coefficient of \eqref{eq:cond3-modX2} then gives
\[
A_\psi=-\frac{2i\sqrt{3}}{9}B_\phi.
\]

The Mellin transform identifies
\[
A_F=-\frac{i}{4\pi^3}L(F,3),
\qquad
B_F=\frac{1}{2\pi^2}L(F,2).
\]
We therefore obtain the following cross-relations.

\begin{theorem}\label{thm:quadratic-twist-critical-values}
	The critical values of the quadratic-twist pair $\phi$ and $\psi=\phi\otimes\chi_{-3}$ satisfy
	
\begin{equation}\label{eq:quadratic-twist-cross-relations}
	L(\psi,3)=\frac{4\pi}{3\sqrt{3}}L(\phi,2),
	\qquad
	L(\phi,3)=\frac{\pi\sqrt{3}}{10}L(\psi,2).
\end{equation}	
\end{theorem}
\begin{remark} The Fricke relations also give $$L(\psi,3)=\frac{\pi^2}{9}L(\psi,1),\qquad L(\phi,3)=\frac{\pi^2}{3}L(\phi,1).$$ The second identity is already contained in Zhou's analysis of $f_{4,6}=\phi$ \cite{MR3798884}.
 \end{remark}

\begin{remark}
	This example follows the same general mechanism as the conductor-$4$ pair, but with a different cusp geometry. The conductor-$3$ twist leads to a two-dimensional translation packet, while Fricke conjugation produces a small slash-stable system. The relevant cusp orbit forms an ideal quadrilateral rather than an ideal triangle, and the resulting period-polynomial identity yields the cross-relations among the critical values of $\phi$ and $\psi$.
\end{remark}

\subsection{Connection with elliptic and Bessel moments}\label{sec:bessel-moments}

The quadratic-twist relations of Theorem~\ref{thm:quadratic-twist-critical-values} also have a natural interpretation in terms of special-function moments. Indeed,
$$\phi(\tau)=\eta(\tau)^2\eta(2\tau)^2\eta(3\tau)^2\eta(6\tau)^2$$
is the weight-$4$, level-$6$ modular form denoted by $f_{4,6}$ in Zhou's paper \cite{MR3798884}. Zhou proved the Bessel-moment evaluations
$$\operatorname{IKM}(2,4;1)=\int_0^\infty I_0(t)^2K_0(t)^4t\,dt=\frac{\pi^2}{2}L(\phi,1)=\frac{3}{2}L(\phi,3)\,,$$
and
$$\frac{3}{\pi^2}\operatorname{IKM}(1,5;1)=\operatorname{IKM}(3,3;1)=\frac{3}{2}L(\phi,2)\,,$$
where
$$\operatorname{IKM}(a,b;1)=\int_0^\infty I_0(t)^aK_0(t)^b t\,dt\,,$$
and $I_0$ and $K_0$ denote the modified Bessel functions.

Theorem~\ref{thm:quadratic-twist-critical-values} gives an additional interpretation of these known Bessel moments in terms of critical values of the quadratic twist $\psi=\phi\otimes\chi_{-3}$. From
$$L(\phi,3)=\frac{\pi\sqrt{3}}{10}L(\psi,2)$$
and Zhou's evaluation of $\operatorname{IKM}(2,4;1)$, we obtain
$$\operatorname{IKM}(2,4;1)=\frac{3\pi\sqrt{3}}{20}L(\psi,2)\,.$$
Similarly, the relation
$$L(\psi,3)=\frac{4\pi}{3\sqrt{3}}L(\phi,2)$$
together with Zhou's evaluations at $L(\phi,2)$ gives
$$\operatorname{IKM}(3,3;1)=\frac{9\sqrt{3}}{8\pi}L(\psi,3)$$
and
$$\operatorname{IKM}(1,5;1)=\frac{3\pi\sqrt{3}}{8}L(\psi,3)\,.$$
Thus the two cross-relations of Theorem~\ref{thm:quadratic-twist-critical-values} connect the Bessel moments occurring in Zhou's formulas with critical values of the quadratic twist $\psi$.

This example belongs to a broader circle of relations among Bessel moments, complete elliptic integrals, Feynman integrals, and modular $L$-values. Elliptic integral evaluations of Bessel moments were studied systematically by Bailey, Borwein, Broadhurst, and Glasser \cite{MR2450513}. Broadhurst subsequently investigated relations among Feynman integrals, $L$-series, and Kloosterman moments \cite{MR3573666}, while Zhou developed analytic and modular methods for proving a number of the corresponding Bessel-moment identities \cite{MR3798884}. Related connections between moments of complete elliptic integrals, lattice sums, and critical values of modular $L$-functions were developed by Wan, by Rogers, Wan, and Zucker, and by Wan and Zucker \cite{MR2845511,MR3338042,MR3490555}.

From this point of view, Theorem~\ref{thm:quadratic-twist-critical-values} supplies an additional piece of structure. Zhou's formulas above concern critical values of the level-$6$ form $\phi$, whereas the modular-symbol relations obtained here connect these values directly with critical values of its quadratic twist $\psi$. Consequently, an independent elliptic integral or Bessel-moment representation of the corresponding critical values of $\psi$ would convert these cross-relations into identities between different special-function moments.

This leads naturally to the question of whether the cross-relations in \eqref{eq:quadratic-twist-cross-relations} can also be understood directly on the special-function side. Since the twist is governed by the character $\chi_{-3}$, one may ask whether the passage from $\phi$ to $\psi$ is reflected in a corresponding transformation of the associated elliptic integral parametrizations. Alternatively, independent Bessel-moment representations for the critical values of $\psi$ would turn the cross-relations above into explicit identities between different Bessel moments.

We do not pursue these questions further here. The purpose of the present discussion is to point out that the period-polynomial and modular-symbol relations obtained above fit naturally into the existing theory connecting modular $L$-values with elliptic and Bessel moments. A systematic investigation of the special-function representations associated with the quadratic twist $\psi$ is left for future research.

\section{A modular-symbol framework}
\label{sec:general-framework}

The preceding examples share a common structure: the relevant modular forms lie in a small finite-dimensional space that is stable under suitable slash operations. We briefly record the resulting modular-symbol framework.

Let
$$\mathbf F=(F_1,\ldots,F_r)^T$$
be a vector of cusp forms of weight $\kappa$ satisfying
$$\mathbf F|_\kappa\gamma=\rho(\gamma)\mathbf F$$
for the relevant modular transformations $\gamma$. For cusps $\alpha$ and $\beta$, define
$$\mathbf R_{\alpha,\beta}(X)=\int_\alpha^\beta\mathbf F(\tau)(\tau-X)^{\kappa-2}\,d\tau\,.$$
In particular,
$$\mathbf R(X)=\mathbf R_{0,\infty}(X)$$
has coefficients expressed in terms of critical $L$-values through the Mellin transform.

The modular-symbol relation
$$\mathbf R_{\alpha,\delta}(X)=\mathbf R_{\alpha,\beta}(X)+\mathbf R_{\beta,\delta}(X)$$
allows paths to be decomposed through intermediate cusps. If
$$\gamma=\begin{pmatrix}a&b\\c&d\end{pmatrix}\in\GL_2^+(\mathbb R),
\qquad D=\det(\gamma),$$
is one of these transformations, then a change of variables gives
$$\mathbf R_{\gamma\alpha,\gamma\beta}(X)=D^{1-\frac{\kappa}{2}}\rho(\gamma)(a-cX)^{\kappa-2}\mathbf R_{\alpha,\beta}\left(\frac{dX-b}{a-cX}\right)\,.$$
Thus, when the sides of a closed cusp polygon are related by modular transformations, the corresponding modular-symbol integrals can be rewritten in terms of the same finite collection of period polynomials. The resulting cusp decomposition therefore produces a finite system of linear relations among their coefficients.

\begin{remark}
	The method separates two related tasks. The first is geometric: one seeks a small slash-stable system of modular forms together with a closed cusp polygon whose sides are related by modular transformations, so that the corresponding modular-symbol integrals close within the same finite-dimensional system. The second is arithmetic: Fricke involutions, twisting identities, Hecke symmetries, or other relations may be used to reduce the resulting period system. The Mellin transform then converts these period relations into relations among critical $L$-values.
\end{remark}

\begin{remark}
	For a twist by a Dirichlet character of conductor $m$, the additive description of twisting introduces translations by fractions $\frac{a}{m}$ and hence additional cusps. This naturally leads to enlarging the slash-stable system by suitable twists or oldforms, as in the quadratic-twist examples above.
\end{remark}

\section*{Funding}

The author declares that no funds, grants, or other support were received during the preparation of this manuscript.	

\section*{Declaration of generative AI and AI-assisted technologies}

During the preparation of this work, the author used ChatGPT (OpenAI)
to assist with mathematical exploration, proof organization, and manuscript
preparation. All mathematical statements, proofs, computations, and
references were independently checked and verified by the author. The
author reviewed and edited the resulting content and takes full
responsibility for the content of the article.
\nocite{*}
\bibliographystyle{elsarticle-num}
\bibliography{references}

@article {MR3338042,
    AUTHOR = {Rogers, M. and Wan, J. G. and Zucker, I. J.},
     TITLE = {Moments of elliptic integrals and critical {$L$}-values},
   JOURNAL = {Ramanujan J.},
  FJOURNAL = {Ramanujan Journal. An International Journal Devoted to the
              Areas of Mathematics Influenced by Ramanujan},
    VOLUME = {37},
      YEAR = {2015},
    NUMBER = {1},
     PAGES = {113--130},
      ISSN = {1382-4090,1572-9303},
   MRCLASS = {11F03 (11M41 33C20 33C75 33E05)},
  MRNUMBER = {3338042},
       DOI = {10.1007/s11139-014-9584-5},
}

@article {MR3231318,
	AUTHOR = {Zhou, Yajun},
	TITLE = {Legendre functions, spherical rotations, and multiple elliptic
	integrals},
	JOURNAL = {Ramanujan J.},
	FJOURNAL = {Ramanujan Journal. An International Journal Devoted to the
	Areas of Mathematics Influenced by Ramanujan},
	VOLUME = {34},
	YEAR = {2014},
	NUMBER = {3},
	PAGES = {373--428},
	ISSN = {1382-4090,1572-9303},
	MRCLASS = {33C05 (33C55 33C75 33E05 44A15)},
	MRNUMBER = {3231318},
	MRREVIEWER = {S.\ Bhargava},
	DOI = {10.1007/s11139-013-9502-2},

}

@article {MR2845511,
	AUTHOR = {Wan, James G.},
	TITLE = {Moments of products of elliptic integrals},
	JOURNAL = {Adv. in Appl. Math.},
	FJOURNAL = {Advances in Applied Mathematics},
	VOLUME = {48},
	YEAR = {2012},
	NUMBER = {1},
	PAGES = {121--141},
	ISSN = {0196-8858,1090-2074},
	MRCLASS = {33C75 (33C20 60G50)},
	MRNUMBER = {2845511},
	MRREVIEWER = {Daniele\ Ritelli},
	DOI = {10.1016/j.aam.2011.04.007},
}

@article {MR3490555,
	AUTHOR = {Wan, J. G. and Zucker, I. J.},
	TITLE = {Integrals of {$K$} and {$E$} from lattice sums},
	JOURNAL = {Ramanujan J.},
	FJOURNAL = {Ramanujan Journal. An International Journal Devoted to the
	Areas of Mathematics Influenced by Ramanujan},
	VOLUME = {40},
	YEAR = {2016},
	NUMBER = {2},
	PAGES = {257--278},
	ISSN = {1382-4090,1572-9303},
	MRCLASS = {11M06 (11F03 33C20 33C75 33E05)},
	MRNUMBER = {3490555},
	MRREVIEWER = {Alia\ Hamieh},
	DOI = {10.1007/s11139-015-9710-z},

}

@misc{SomosEta,
	author       = {Michael Somos},
	title        = {Dedekind Eta Function Product Identities},
	howpublished = {Archived website},
	note         = {Archived version of the former \texttt{eta.math.georgetown.edu} website, available at
	\url{https://web.archive.org/web/20190709150645/https://eta.math.georgetown.edu/index.html},
	accessed August 20, 2026}
}

@article{glaisher1907representations,
	title={On the representations of a number as the sum of two, four, six, eight, ten, and twelve squares},
	author={Glaisher, JWL},
	journal={Quart. J. Math},
	volume={38},
	pages={1--62},
	year={1907}
}

@article {MR314846,
	AUTHOR = {Manin, Yuri Ivanovich},
	TITLE = {Parabolic points and zeta functions of modular curves},
	JOURNAL = {Izv. Akad. Nauk SSSR Ser. Mat.},
	FJOURNAL = {Izvestiya Akademii Nauk SSSR. Seriya Matematicheskaya},
	VOLUME = {36},
	YEAR = {1972},
	PAGES = {19--66},
	ISSN = {0373-2436},
	MRCLASS = {14G10 (10D15 14H25)},
	MRNUMBER = {314846},
}

@article {MR345909,
	AUTHOR = {Manin, Yuri Ivanovich},
	TITLE = {Periods of cusp forms, and {$p$}-adic {H}ecke series},
	JOURNAL = {Mat. Sb. (N.S.)},
	FJOURNAL = {Matematicheski\u i\ Sbornik. Novaya Seriya},
	VOLUME = {92(134)},
	YEAR = {1973},
	PAGES = {378--401},
	ISSN = {0368-8666},
	MRCLASS = {10D05 (12A70)},
	MRNUMBER = {345909},
	}

@article {MR463119,
	AUTHOR = {Shimura, Goro},
	TITLE = {On the periods of modular forms},
	JOURNAL = {Math. Ann.},
	FJOURNAL = {Mathematische Annalen},
	VOLUME = {229},
	YEAR = {1977},
	NUMBER = {3},
	PAGES = {211--221},
	ISSN = {0025-5831,1432-1807},
	MRCLASS = {10D15},
	MRNUMBER = {463119},
	DOI = {10.1007/BF01391466},
}

@article {MR3848417,
	AUTHOR = {Li, Wen-Ching Winnie and Long, Ling and Tu, Fang-Ting},
	TITLE = {Computing special {$L$}-values of certain modular forms with
	complex multiplication},
	JOURNAL = {SIGMA Symmetry Integrability Geom. Methods Appl.},
	FJOURNAL = {SIGMA. Symmetry, Integrability and Geometry. Methods and
	Applications},
	VOLUME = {14},
	YEAR = {2018},
    PAGES={090},
	ISSN = {1815-0659},
	MRCLASS = {11F11 (11F67 11M36 33C05)},
	MRNUMBER = {3848417},
	MRREVIEWER = {Wenjun\ Ma},
	DOI = {10.3842/SIGMA.2018.090},
}

@article {MR3798884,
	AUTHOR = {Zhou, Yajun},
	TITLE = {Wick rotations, {E}ichler integrals, and multi-loop {F}eynman
	diagrams},
	JOURNAL = {Commun. Number Theory Phys.},
	FJOURNAL = {Communications in Number Theory and Physics},
	VOLUME = {12},
	YEAR = {2018},
	NUMBER = {1},
	PAGES = {127--192},
	ISSN = {1931-4523,1931-4531},
	MRCLASS = {81S40 (11F03 58D30)},
	MRNUMBER = {3798884},
	MRREVIEWER = {Emilio\ Elizalde},
	DOI = {10.4310/CNTP.2018.v12.n1.a5},

}

@article {MR2450513,
	AUTHOR = {Bailey, David H. and Borwein, Jonathan M. and Broadhurst,
	David and Glasser, M. L.},
	TITLE = {Elliptic integral evaluations of {B}essel moments and
	applications},
	JOURNAL = {J. Phys. A},
	FJOURNAL = {Journal of Physics. A. Mathematical and Theoretical},
	VOLUME = {41},
	YEAR = {2008},
	NUMBER = {20},
	PAGES = {205203},
	ISSN = {1751-8113,1751-8121},
	MRCLASS = {33C10 (33E05 33F05)},
	MRNUMBER = {2450513},
	MRREVIEWER = {Subuhi\ Khan},
	DOI = {10.1088/1751-8113/41/20/205203},

}

@article {MR3573666,
	AUTHOR = {Broadhurst, David},
	TITLE = {Feynman integrals, {L}-series and {K}loosterman moments},
	JOURNAL = {Commun. Number Theory Phys.},
	FJOURNAL = {Communications in Number Theory and Physics},
	VOLUME = {10},
	YEAR = {2016},
	NUMBER = {3},
	PAGES = {527--569},
	ISSN = {1931-4523,1931-4531},
	MRCLASS = {11G40 (81Q30)},
	MRNUMBER = {3573666},
	MRREVIEWER = {Ioulia\ N.\ Baoulina},
	DOI = {10.4310/CNTP.2016.v10.n3.a3},
	
}

@book {MR1641658,
	AUTHOR = {Borwein, Jonathan M. and Borwein, Peter B.},
	TITLE = {Pi and the {AGM}},
	SERIES = {Canadian Mathematical Society Series of Monographs and
	Advanced Texts},
	VOLUME = {4},
	NOTE = {A study in analytic number theory and computational complexity,
	Reprint of the 1987 original,
	A Wiley-Interscience Publication},
	PUBLISHER = {John Wiley \& Sons, Inc., New York},
	YEAR = {1998},
	PAGES = {xvi+414},
	ISBN = {0-471-31515-X},
	MRCLASS = {11Y60 (11B65 68Q25)},
	MRNUMBER = {1641658},
}
\end{document}